\documentclass{article}

\usepackage{amsmath,amsfonts,amssymb,amsthm}

\newtheorem{thm}{Theorem}[section]
\newtheorem{prop}[thm]{Proposition}
\newtheorem{cor}[thm]{Corollary}
\newtheorem*{thmA}{Theorem A}
\newtheorem*{thmB}{Theorem B}

\newcommand{\spa}[1]{\langle #1 \rangle}

\newcommand{\Ker}{\operatorname{Ker}}

\newcommand{\rank}{\operatorname{rank}}

\begin{document}

\title{Intersection of subspaces in $A^2$\\for a three-dimensional division algebra $A$\\ over a finite field}

\author{Daisuke Tambara\\
Hirosaki University\\
e-mail: tambara@hirosaki-u.ac.jp}

\date{}

\maketitle

\begin{abstract}
Let $A$ be a three-dimensional nonassociative division algebra over a finite field $F$.
Let $A$ act on the space $A^2=A\oplus A$ by left multiplication. For a nonzero vector $v\in A^2$ we have a three-dimensional subspace $Av$ in $A^2$. This paper concerns about possible dimension of  intersections $Av\cap Av'$ for $v,v'\in A^2$.
One of our results is that there exists a two-dimensional intersection if and only if $A$ is isotopic to a commutative algebra. We use a classical theorem that $A$ is a twisted field of Albert.

\end{abstract}

\section*{Introduction}

Finite nonassociative division algebras have long been studied since Dickson's work \cite{Dic},
as seen from a survey by Cordero and Wene \cite{Cord}.
Albert, in his study of the relationship between finite division algebras and finite projective planes \cite{Alb4}, considered the left vector spaces $A^n$ over a division algebra $A$. He noticed that a basic property of ordinary vector spaces does not hold for nonassociative algebras 
\cite[Section 5]{Alb4}.
Compared with abundant works on projective planes over finite division algebras, little attention seems to have been attracted to vector spaces over them.
In this paper we shall make a closer look at the nature of subspaces of $A^n$ in the special case where $n=2$ and $A$ is three-dimensional over a finite field.

Let $A$ be a division algebra over a field $F$.
Let $A$ act on the space $A^2$ by left multiplication: $a(x,y)=(ax,ay)$.
For an element $v\in A^2$ we have a subspace $Av=\{av\mid a\in A\}$, which has the same dimension as $A$ unless $v=0$.
The paper is concerned with intersection of the subspaces $Av$ for $v\in A^2$.
When $A$ is a field, $A^2$ being an ordinary vector space over $A$, two different spaces $Av$ and $Av'$ intersect trivially. 
But when $A$ is nonassociative, nontrivial intersection may happen.
We are interested in what dimension $Av\cap Av'$ can take.
We answer the question in a special case below. 

Assume that  $F$ is a finite field and $A$ is a three-dimensional nonassociative division algebra over $F$. 
We first decide when $Av=Av'$ for $v,v'\in A^2$.
We call $v=(x,y)\in A^2$ a {\itshape nondegenerate vector} if $x, y$ are linearly independent over $F$.

\begin{thmA}
For any nondegenerate vectors $v,v'\in A^2$ we have $Av=Av'$ if and only if $Fv=Fv'$.

\end{thmA}

Our second result characterizes  algebras $A$ admitting a two-dimensional intersection $Av\cap Av'$.
Recall that two algebras $A$ and $A'$ are said to be {\itshape isotopic} if there exist linear isomorphisms $f,g,h\colon A\to A'$ such that $f(a)g(b)=h(ab)$ for all $a,b\in A$.

\begin{thmB}
There exist $v,v'\in A^2$ such that $\dim(Av\cap Av')=2$ if and only if $A$ is isotopic to a commutative algebra.

\end{thmB}

In proving these we use a theorem that any three-dimensional nonassociative division algebras over a finite field is an Albert twisted field, due to Kaplansky and Menichetti (\cite{Kap1}, \cite{Kap2}, \cite{Men}). Also we follow Kaplansky's \lq algebraically closed style\rq, in which the base field is extended to its algebraic closure.

Based on these results we compute the number of  subspaces $Av'$ complementary to a given subspace $Av$ (Propositions 9.12 and 9.16).

The paper is organized as follows.
In Section 1 we review basic facts about Albert's twisted fields. 
In Section 2 we consider an algebra obtained from a twisted field by base extension to an algebraic closure, which we here call a {\itshape split Albert algebra}.
In Section 3 we prove a split version of 
Theorem A, from which we deduce the theorem in Section 4.
In Section 5 we prove the \lq if\rq\ part of Theorem B.
In Section 6 we  prepare  some propositions on intersections $Av\cap Av'$. 
We analyse the situation in which two-dimensional intersections occur
for a split Albert algebra in Section 7, thereby
deduce the \lq only if\rq\ part of Theorem B in Section 8.
 Finally we compute the number of complementary subspaces in Section 9. 
 
By a division algebra we mean a vector space $A$ over a field equipped with a bilinear map $A\times A\to A\colon (x,y)\mapsto xy$ such that for every nonzero $a\in A$ the left multiplication $L_a\colon x\mapsto ax$ and the right multiplication $R_a\colon x\mapsto xa$  are bijections. We do not put the axiom of an identity element, so this is a pre-semifield in Knuth's terminology (\cite{Knu}). Sometimes we refer to $A$ as 
$(A,m)$, denoting the bilinear map $A\times A\to A$ by $m$.

\section{Twisted fields}

We review here the construction of a twisted field associated with a cubic extension, and then
describe its multiplication after  the base extension  to a splitting field.

Let $K/F$ be a cyclic cubic extension and $\sigma$  a generator of its Galois group.
Let $N$ denote the norm map $K\to F$ for the extension.
Let $c\in K$ be an element not equal to $x^{\sigma}x^{-1}$ for any $x\in K^{\times}$. 
This amounts to requiring that $N(c)\ne 1$.

Define a map $\mu\colon K\times K\to K$ by
$\mu(x,y)=x y^{\sigma}-c x^{\sigma}y$.
The pair $(K, \mu)$ is a division algebra over $F$.
The map $\mu$ can be modified  by linear automorphisms  into a multiplication admitting an identity element. The resulting unital division algebra is called a twisted field (\cite{Alb2}, \cite{Alb3}).
As an identity element is irrelevant to our work, we call less strictly the algebra $(K,\mu)$ the twisted field associated with the triple $(K/F, \sigma, c)$.

If $c=-1$, $(K,\mu)$ is commutative.

Two algebras $A$ and $A'$ are said to be {\itshape isotopic} if there exist linear isomorphisms $f,g,h\colon A\to A'$ such that $f(a)g(b)=h(ab)$ for all $a,b\in A$ (\cite{Alb1}).

Let $c'\in K$ with $N(c')\ne 1$ and  let $(K, \mu')$ be the corresponding twisted field. 
If $N(c)=N(c')$, then $(K,\mu)$ and $(K, \mu')$ are isotopic.
Indeed, write $c'/c=a^{\sigma}a^{-1}$ with $a\in K^{\times}$;
Then $a\mu'(x,y)=\mu(ax,y)$ for all $x,y\in K$.

In particular, if $N(c)=-1$, then $(K,\mu)$ is isotopic to a commutative algebra.

It is known that any three-dimensional nonassociative unital division algebra over a finite field is isomorphic to a twisted field in the strict sense. This theorem was conjectured and partly proved  by Kaplansky, and completed by Menichetti
(\cite{Kap2}, \cite{Men}).
As we are ignoring  unitality, we should phrase this theorem in terms of isotopy rather than isomorphism:

\begin{thm}
Let $F$ be a finite field and $A$ a three-dimensional nonassociative division algebra over $F$.
Then $A$ is isotopic to a twisted field $(K, \mu)$ for a cubic extension $K/F$ and an element $c\in K$.
\end{thm}

Following Kaplansky's \lq algebraically closed style\rq (\cite{Kap1}), we discuss splitting of twisted fields.

Let $(K, \mu)$ be the twisted field associated with $(K/F,\sigma,c)$.
Since $K/F$ is a Galois extension,
we have an algebra isomorphism
$\omega\colon K\otimes K\to K^3$
taking 
$x\otimes y$ to $(xy,xy^{\sigma}, xy^{\sigma^2})$,
where the algebra structures of $K\otimes K$ and $K^3$ are the standard ones.
This transforms the automorphism $1\otimes \sigma$ of $K\otimes K$ into 
an automorphism $\rho$ of $K^3$ given by
$\rho\colon (x_i)\mapsto (x_{i+1})$,
where the index is taken modulo 3.
Put $c_i=c^{\sigma^i}$ and $\gamma=(c_i)_i\in K^3$.
Then $\omega$ transforms the multiplication $1\otimes \mu$ on $K\otimes K$ into a multiplication $\nu$ on $K^3$ given by
$\nu(\xi, \eta)=\xi \eta^{\rho}-\gamma \xi^{\rho}\eta$,
or in coordinates 
$\nu((x_i), (y_i))=(x_iy_{i+1}-c_ix_{i+1}y_i)$.
Let $(e_0,e_1,e_2)$ be the standard basis of $K^3$.
Then
$$
\nu (e_i, e_j)=
\begin{cases} e_i \quad &\text{if $j=i+1$,}\\
-c_je_j &\text{if $i=j+1$,}\\
0 &\text{otherwise.}
\end{cases}
$$

In summary

\begin{prop} 
After the scalar extension $K/F$, the twisted field $(K, \mu)$ becomes isomorphic to the $K$-algebra $(K^3, \nu)$ with multiplication $\nu$ given by the above formula.
\end{prop}

The isomorphism $\omega$ transforms the automorphism $\sigma\otimes 1$ of $K\otimes K$ into 
an automorphism $\lambda$ of $K^3$ given by
$\lambda\colon (x_i)\mapsto (x_{i-1}^{\sigma})$.
This is not $K$-linear but semi-linear relative to $\sigma$.
It permutes the basis as
$\lambda\colon e_i\mapsto e_{i+1}$.

\section{Split Albert algebras}

We consider here an algebra of the form $(K^3,\nu)$ of Section 1 a little more generally, the base field set down to $F$.

Let $U$, $V$, $W$ be three-dimensional vector spaces over $F$ respectively having bases
$(\alpha_i)$, $(\beta_i)$, $(\gamma_i)$, where  the index $i$ runs through $0,1,2$.
Let $d_0,d_1,d_2\in F^{\times}$.  Define a bilinear map
$\phi\colon U\times V\to W$ by
$$
\phi(\alpha_i,\beta_i)=0,\;
\phi(\alpha_i,\beta_{i+1})=\gamma_{i+2},\;
\phi(\alpha_i,\beta_{i+2})=d_{i+1}\gamma_{i+1},
$$
where the index is taken modulo 3.
Put $d=d_0d_1d_2$.
We assume $d\ne -1$ throughout.
Let us call $\phi$ a {\itshape split Albert algebra}.

The multiplication $\nu\colon K^3\times K^3\to K^3$ of Section 1 is a special case of $\phi$ where $F$ is set as $K$ and
$U=V=W=K^3$, 
$\alpha_i=e_i, \; \beta_i=e_i,\; \gamma_i=e_{i-1}$,
and
$d_i=-c^{\sigma^i}$.
Note that $d=-N(c)$.

When only a single $\phi$ is concerned, we write $\phi(x,y)=xy$.

For $x\in U$ let $L_x\colon V\to W$ be the map $y\mapsto xy$. 
For $y\in V$ let $R_y\colon U\to W$ be the map $x\mapsto xy$.

Let $x=x_0 \alpha_0+x_1\alpha_1+x_2\alpha_2$.
Then
$x\beta_i=x_{i-1}\gamma_{i+1}+x_{i+1}d_{i-1}\gamma_{i-1}$.
Relative to the present bases the linear map $L_x\colon V\to W$ is represented by a matrix
$$
\begin{pmatrix}
0 & d_0 x_2 & x_1 \\
x_2 & 0 & d_1 x_0 \\
d_2 x_1 & x_0 & 0 
\end{pmatrix}.
$$
Its determinant is $(1+d)x_0x_1x_2$.
Since $d\ne -1$, $L_x$ is invertible if and only if none of $x_i$ is zero, in which case call $x$ a {\itshape regular element}.
On the other hand, if $x_0=0$ and $(x_1,x_2)\ne (0,0)$, then $\Ker L_x$ is a one-dimensional space spanned by $x_1\beta_1-d_0x_2\beta_2$. Similarly, for any nonzero nonregular element $x\in U$ one sees that $\Ker L_x$ is one-dimensional. It follows that if $y,y'\in V$ are linearly independent over $F$,  then the map $U\to W^2\colon x\mapsto (xy,xy')$ is injective.

Let 
$y=y_0\beta_0+y_1\beta_1+y_2\beta_2$.
Then 
$\alpha_iy
=y_{i+1}\gamma_{i-1}+y_{i-1}d_{i+1}\gamma_{i+1}$.
Relative to the present bases the linear map $R_y\colon U\to W$ is represented by a matrix
$$
\begin{pmatrix}
0       & y_2    & d_0y_1\\
d_1 y_2 & 0      & y_0\\
y_1     & d_2y_0 & 0 
\end{pmatrix}
$$
with determinant 
$(1+d)y_0y_1y_2$.
So $R_y$ is invertible if and only if none of $y_i$ is zero. In this case we call $y$ a regular element.
Then $R_y^{-1}$ is represented by a matrix
$$
\frac{1}{1+d_0d_1d_2}
\begin{pmatrix}
-d_2\frac{y_0}{y_1y_2} & d_0d_2\frac{1}{y_2} & \frac{1}{y_1} \\
\frac{1}{y_2} & -d_0\frac{y_1}{y_0y_2} & d_0d_1\frac{1}{y_0} \\
d_1d_2\frac{1}{y_1} & \frac{1}{y_0} & -d_1\frac{y_2}{y_0y_1}
\end{pmatrix}.
$$

Kaplansky called an algebra $A$ a left Dickson algebra if the determinant of the left   multiplication $L_x$ is a product of linearly independent linear forms on  $x$ \cite{Kap1}. 
The right counterpart is called a right Dickson algebra.
Thus a split Albert algebra is a left and right Dickson algebra as long as $U=V=W$.

We discuss some isomorphisms between split Albert algebras.
For any bilinear maps $l\colon X\times Y\to Z$ and $l'\colon X'\times Y'\to Z'$, an isomorphism $l\to l'$ means a triple $(f,g,h)$ of linear isomorphisms $f\colon X\to X'$, $g\colon Y\to Y'$, $h\colon Z\to Z'$ such that $hl=l'(f\times g)$.
An isotopy between algebras is a special case where $X=Y=Z$, $X'=Y'=Z'$.

We use notation $\phi=\phi_{d_0,d_1,d_2}$ to make clear the dependence on $d_i$.
Given $r_i, s_i\in F^{\times}$ let $f\colon U\to U$, $g\colon V\to V$, $h\colon W\to W$ be 
respectively the linear maps
$$
\alpha_i\mapsto 
r_i\alpha_i,\quad
\beta_i\mapsto s_i\beta_i,\quad
\gamma_i\mapsto r_{i+1}s_{i+2}\gamma_i.
$$
Then $(f,g,h)$ gives an isomorphism $\phi_{d_0,d_1,d_2}\to \phi_{d_0',d_1',d_2'}$, 
where
$$
d_{i}'=\frac{r_{i+1}}{r_{i-1}}\frac{s_{i-1}}{s_{i+1}} d_{i}.
$$
Note that 
$d_0d_1d_2=d_0'd_1'd_2'$.

The linear isomorphisms 
$$
\alpha_i\mapsto \alpha_{i+1},\quad
\beta_i\mapsto \beta_{i+1},\quad
\gamma_i\mapsto \gamma_{i+1}
$$
give an isomorphism $\phi_{d_0,d_1,d_2}\to \phi_{d_2,d_0,d_1}$.

The linear isomorphisms
$$
\alpha_i\mapsto \alpha_{1-i},\quad
\beta_i\mapsto \beta_{1-i},\quad
\gamma_i\mapsto  d_i^{-1}\gamma_{1-i}
$$
give an isomorphism $\phi_{d_0,d_1,d_2}\to \phi_{d_1^{-1},d_0^{-1},d_2^{-1}}$.

\begin{prop}
Let $y, y'\in V$ be regular elements.
The characteristic polynomial of $R_{y'}^{-1}R_y$ is given by
$$
\det(XI-R_{y'}^{-1}R_y)=(X-\frac{y_0}{y'_0})(X-\frac{y_1}{y'_1})(X-\frac{y_2}{y'_2}).
$$
\end{prop}

\begin{proof}
Relative to the basis $(\alpha_i)$, $R_{y'}^{-1}R_y$ is represented by a matrix
$$
\frac{1}{1+d_0d_1d_2}
\begin{pmatrix}
d_0d_1d_2t_2+t_1
& d_2(t_0-t_2)\frac{y'_0}{y'_1}
& d_2d_0(t_0-t_1)\frac{y'_0}{y'_2}\\
d_0d_1(t_1-t_2)\frac{y'_1}{y'_0}
& d_0d_1d_2t_0+t_2
& d_0(t_1-t_0)\frac{y'_1}{y'_2} \\
d_1(t_2-t_1)\frac{y'_2}{y'_0} 
& d_1d_2(t_2-t_0)\frac{y'_2}{y'_1}
& d_0d_1d_2t_1+t_0
\end{pmatrix}
$$
with 
$t_i=\frac{y_i}{y'_i}$.
We know
$$
\det(R_{y'}^{-1}R_y)=\frac{y_0y_1y_2}{y_0'y_1'y_2'}=t_0t_1t_2.
$$
The trace of the above matrix is readily found to be
$t_0+t_1+t_2$.
One of the principal 2-minor of the matrix is
$$
\begin{vmatrix}
d_0d_1d_2t_2+t_1
& d_2(t_0-t_2)\frac{y'_0}{y'_1}\\
d_0d_1(t_1-t_2)\frac{y'_1}{y'_0}
& d_0d_1d_2t_0+t_2
\end{vmatrix}
=(1+d)(dt_0t_2+t_1t_2).
$$
It follows that the sum of the principal 2-minors of the matrix is equal to
$t_0t_1+t_1t_2+t_2t_0$.
This proves the proposition.

\end{proof}

\section{The equation $U(x,y)=U(x',y')$ for a split Albert algebra}

Let $\phi=\phi_{d_0,d_1,d_2}\colon U\times V\to W$ be a split Albert algebra. We regard $\phi$ as multiplication: $\phi(u,v)=uv$.
This induces a bilinear map $U\times V^2\to W^2$: $(u,(x,y))\mapsto (ux,uy)
=u(x,y)$.
We write $U(x,y)=\{u(x,y)\mid u\in U\}$ for $(x,y)\in V^2$. This is a subspace of $W^2$.

Recall that $x=\sum_i x_i\beta_i\in V$ is called a regular element if none of $x_i$ is zero.

The main result of this section is 

\begin{thm}
Let $x,y,x',y'\in V$ be regular elements. Then
$U(x,y)=U(x',y')$ if and only if
$(x',y')=k(x,y)$ or $(y,y')=k(x,x')$ for some $k\in F^{\times}$.
\end{thm}

The theorem will follow from the two propositions below.
Recall from Section 2 that for regular elements $x,x'\in V$ we have a linear map
$R_{x'}^{-1}R_x\colon U\to U$.

\begin{prop}
Let $x,y,x',y'\in V$ be regular elements. Then
$U(x,y)=U(x',y')$ if and only if $R_{x'}^{-1}R_x=R_{y'}^{-1}R_y$.
\end{prop}

\begin{proof}
Suppose $U(x,y)=U(x',y')$. For any $a\in U$ there exists a unique $a'\in U$ such that
$a(x,y)=a'(x',y')$,
namely
$ax=a'x',\; ay=a'y'$.
Then
$a'=R_{x'}^{-1}R_x(a), \; a'=R_{y'}^{-1}R_y(a)$,
hence
$R_{x'}^{-1}R_x(a)=R_{y'}^{-1}R_y(a)$. Thus
$R_{x'}^{-1}R_x=R_{y'}^{-1}R_y$.
The argument can be reversed.
\end{proof}

\begin{prop}
Let $x,y,x',y'\in V$ be regular elements.
We have $R_{x'}^{-1}R_x=R_{y'}^{-1}R_y$ if and only if
$(x',y')=k(x,y)$ or $(y,y')=k(x,x')$ for some $k\in F^{\times}$.
\end{prop}

\begin{proof}
The sufficiency is clear.
Let us prove the necessity.
Suppose $R_{x'}^{-1}R_x=R_{y'}^{-1}R_y$.
Put $s_i=x_i/x'_i$, $t_i=y_i/y'_i$.
By Proposition 2.1 the triples $(s_0,s_1,s_2)$ and $(t_0,t_1,t_2)$ are equal up to permutation.
We divide cases according as the type of the permutation.
By the cyclic symmetry $i\mapsto i+1$ of the indices and the symmetry $(x,x')\leftrightarrow (y,y')$, it is enough to consider  three cases. 
Case I: $(s_0,s_1,s_2)=(t_0,t_1,t_2)$;
Case II: $(s_2, s_0, s_1)=(t_0, t_1, t_2)$;
Case III: $(s_0, s_2, s_1)=(t_0, t_1, t_2)$.

\medskip
Case I: $(s_0,s_1,s_2)=(t_0,t_1,t_2)$.
Then $x_i/x_i'=y_i/y_i'$ and $y_i/x_i=y_i'/x_i'$.

(i) Case where $t_0,t_1,t_2$ are not all equal.
By the cyclic symmetry we may assume $t_0\ne t_1, t_0\ne t_2$.
The matrix of $R_{y'}^{-1}R_y$ is exhibited in Section 2.
Its rows and columns will be indexed by $0,1,2$.
Comparing the $(0,1)$-entry and $(0,2)$-entry of the matrices of $R_{x'}^{-1}R_x$ and $R_{y'}^{-1}R_y$, we have
$$
(s_0-s_2)\frac{x_0'}{x_1'}=(t_0-t_2)\frac{y_0'}{y_1'},\quad
(s_0-s_1)\frac{x_0'}{x_2'}=(t_0-t_1)\frac{y_0'}{y_2'}.
$$
Since $s_0-s_2=t_0-t_2\ne 0$, $s_0-s_1=t_0-t_1\ne 0$, it follows that
$$
\frac{x_0'}{x_1'}=\frac{y_0'}{y_1'},\quad
\frac{x_0'}{x_2'}=\frac{y_0'}{y_2'},
$$
hence
$$
\frac{y_0'}{x_0'}=\frac{y_1'}{x_1'}=\frac{y_2'}{x_2'}.
$$
Call this element $k$.
Then 
$y'=k x'$.
We have also $y_i/x_i=k$, 
$y=k x$. 

(ii) Case where $t_0=t_1=t_2$.
Put $t_i=k$. Then $y=ky'$ and $x=kx'$.

\medskip
Case II: $(s_2,s_0,s_1)=(t_0,t_1,t_2)$.
Comparing the diagonal entries of the matrices of $R_{x'}^{-1}R_x$ and $R_{y'}^{-1}R_y$, we have
$$
ds_2+s_1=dt_2+t_1,\quad
ds_0+s_2=dt_0+t_2,\quad
ds_1+s_0=dt_1+t_0.
$$
Substituting $s_i=t_{i+1}$, we have
$$
dt_0+t_2=dt_2+t_1,\quad
dt_1+t_0=dt_0+t_2,\quad
dt_2+t_1=dt_1+t_0.
$$
Thus
\begin{align*}
dt_0-t_1+(1-d)t_2&=0, \tag1\\
dt_1-t_2+(1-d)t_0&=0, \tag2\\
dt_2-t_0+(1-d)t_1&=0. \tag3
\end{align*}
Comparing the off-diagonal entries of the two matrices and substituting $s_i=t_{i+1}$, we have
\begin{align*}
(t_1-t_0)\frac{x_0'}{x_1'}&=(t_0-t_2)\frac{y_0'}{y_1'}, \tag4\\
(t_1-t_2)\frac{x_0'}{x_2'}&=(t_0-t_1)\frac{y_0'}{y_2'}, \tag5\\
(t_2-t_0)\frac{x_1'}{x_0'}&=(t_1-t_2)\frac{y_1'}{y_0'}, \tag6\\
(t_2-t_1)\frac{x_1'}{x_2'}&=(t_1-t_0)\frac{y_1'}{y_2'}, \tag7\\
(t_0-t_2)\frac{x_2'}{x_0'}&=(t_2-t_1)\frac{y_2'}{y_0'}, \tag8\\
(t_0-t_1)\frac{x_2'}{x_1'}&=(t_2-t_0)\frac{y_2'}{y_1'}. \tag9
\end{align*}
By elementary operations one sees that (1)--(3)  are equivalent to equations
\begin{align*}
(1-d+d^2)(t_1-t_2)&=0, \tag{10}\\
t_0-t_1+d(t_1-t_2)&=0. \tag{11}
\end{align*}

(i) Case where $1-d+d^2\ne 0$.
Then $t_0=t_1=t_2$, 
which falls into (ii) of Case I.

(ii) Case where $1-d+d^2=0$.
By (11), if $t_1=t_2$, then $t_0=t_1=t_2$ again.
Assume $t_1\ne t_2$.
Putting (11) into  (5) and (7), we have
$$
(t_1-t_2)\frac{x_0'}{x_2'}=-d(t_1-t_2)\frac{y_0'}{y_2'}, \quad
(t_2-t_1)\frac{x_1'}{x_2'}=-d(t_2-t_1)\frac{y_1'}{y_2'},
$$
hence
$$
\frac{x_0'}{x_2'}=-d\frac{y_0'}{y_2'}, \quad \frac{x_1'}{x_2'}=-d\frac{y_1'}{y_2'}.
$$
These two equations yield
$$
\frac{x_0'}{x_1'}=\frac{y_0'}{y_1'}.
$$
Putting this into (4) and (6), we have
$$
t_1-t_0=t_0-t_2,\quad t_2-t_0=t_1-t_2,
$$
hence
$$
2t_0=t_1+t_2, \quad 2t_2=t_0+t_1.
$$
It follows that
$$
2(t_0-t_2)=t_2-t_0.
$$
If $\text{char}(F)\ne 3$, then
$t_0-t_2=0$, 
so
$t_0=t_1=t_2$,
contrary to the present assumption.
If $\text{char}(F)=3$, the equation  $d^2-d+1=0$ gives $d=-1$, which is to be excluded.

\medskip
Case III: $(s_0, s_2, s_1)=(t_0, t_1, t_2)$.
Comparing the $(2,2)$-entry of the two matrices, we have
$ds_1+s_0=dt_1+t_0$.
In our case this becomes
$dt_2+t_0=dt_1+t_0$. 
This gives $t_1=t_2$. So $s_1=s_2$.
Then $(s_0, s_1, s_2)=(t_0,t_1,t_2)$.
So we are back in Case I.

\medskip
We conclude that
$y=kx, \; y'=kx'$
or $x'=kx, \; y'=ky$
for some $k\in F^{\times}$.
This proves the proposition.
\end{proof}

\section{The equation $Av=Av'$ for a division algebra}

Let $F$ be a finite field. Let $A$ be a three-dimensional nonassociative division algebra over $F$.
We have the left action of $A$ on $A^2$: $a(x,y)=(ax,ay)$.
For $v=(x,y)\in A^2$ we say $v$ is {\itshape nondegenerate} if $x, y$ are linearly independent over $F$, and {\itshape degenerate} otherwise.

\begin{thm} 
Let $v, v'\in A^2$ be nondegenerate vectors. Then
$Av=Av'$ if and only if $Fv=Fv'$.
\end{thm}

\begin{proof}
Let $K/F$ be a cubic extension with generating automorphism $\sigma$.
By Theorem 1.1 $A$ is isotopic to the twisted field $(K, \mu)$ associated with an element $c\in K^{\times}$.
By Proposition 1.2 we have an isomorphism $(f,g,h)$ of bilinear maps from 
the multiplication  map $(K\otimes A)\times (K\otimes A)\to K\otimes A$ to 
a split Albert algebra $\phi_{d_0,d_1,d_2}\colon U\times V\to W$ over $K$.
Write $v=(x,y)$, $v'=(x',y')$.
Suppose $Av=Av'$. Then $U(g(x),g(y))=U(g(x'), g(y'))$.
Since $x\ne 0$ and $A$ is a division algebra, the right multiplication $a\mapsto ax$ on $A$ is invertible. Hence $R_{g(x)}\colon U\to W$ is invertible, namely $g(x)$ is a regular element.
Similarly $g(y), g(x'), g(y')$ are regular elements.
Applying Theorem 3.1 to these elements, we have
$(g(x'), g(y'))=k(g(x),g(y))$ or 
$(g(y), g(y'))=k(g(x), g(x'))$
for some $k\in K^{\times}$.
Since $g$ is an isomorphism, we have
$(x',y')=k(x,y)$ or  $(y,y')=k(x,x')$.
In either case $k\in F^{\times}$.
But the second case would imply that $v, v'$ are degenerate.
Therefore we must have $v'=kv$.
\end{proof}

As for degenerate vectors $v,v'\in A^2$ it is easy to decide when $Av=Av'$. See Proposition 6.3.

\section{Commutative algebras}

\begin{thm} Let $F$ be a finite field. Let $A$ be a three-dimensional nonassociative division algebra over $F$. 
Suppose that $A$ is isotopic to a commutative algebra.
Then there exist $v,v'\in A^2$ such that $\dim(Av\cap Av')=2$.
\end{thm}

\begin{proof}
We may assume that $A$ itself is commutative.
Let $v=(x,y)$ be a nondegenerate vector.
Take $x'\in A-Fx$.
Take $y'\in A$ such that $x'y=xy'$.
By the commutativity
we have
$x'(x,y)=x(x',y')$.
Also
$y'x=yx'$ and $y'y=yy'$, hence
$y'(x,y)=y(x',y')$.
Put $v'=(x',y')$.
We have 
$x'v=xv', \; y'v=yv'$.
So 
$\spa{x',y'}v=\spa{x,y}v'$.
This is two-dimensional and contained in $Av\cap Av'$.
Since $x,x'$ are independent, so are $v, v'$.
Therefore $Av\ne Av'$ by Theorem 4.1.
It follows that $Av\cap Av'$ is two-dimensional and coincides with 
$\spa{x',y'}v=\spa{x,y}v'$. 
\end{proof}

\section{Nontrivial intersection of $Av$} 

 We make here some preparations for the remaining part of Theorem B of Introduction.
Propositions 6.1--6 state elementary facts.
Propositions 6.7--9 are facts peculiar to three-dimensional nonassociative division algebras over a finite field.

Let $A$ be an algebra over $F$.
We have an operation $A\times A^2\to A^2$: $a(x,y)=(ax,ay)$.
For any $v\in A^2$ we have a subspace $Av=\{av\mid a\in A\}\subset A^2$. 
We say $v$ is {\itshape regular} if $av=0$ implies $a=0$.
In this case the map $a\mapsto av$ gives a linear isomorphism $A\to Av$.

\begin{prop}
(i) Suppose that $v,v'\in A^2$ are both regular.
Then we have an isomorphism $Av\cap Av'\cong \{(a,a')\in A^2\mid av=a'v'\}$
given by the correspondence $av\leftrightarrow (a,a')$.

(ii) 
Let $P\in GL_2(F)$.
We have a linear isomorphism $A^2\to A^2$ given by $v\mapsto vP$, where $v$ is regarded as a row vector.
Let $v,v'\in A^2$ and put $w=vP$, $w'=v'P$.
Then the above isomorphism induces an isomorphism $Av\cap Av'\cong Aw\cap Aw'$.
And we have 
$$
\{(a,a')\in A^2\mid av=a'v'\}=\{(a,a')\in A^2\mid aw=a'w'\}.
$$

(iii)
Let $Q\in GL_2(F)$.
Let $v,v'\in A^2$ and put
$$
\begin{pmatrix} u \\ u'
\end{pmatrix}
=Q\begin{pmatrix} v \\ v'
\end{pmatrix}.
$$
Then we have a linear isomorphism
$$
\{(a,a')\in A^2\mid av=a'v'\}\cong \{(b,b')\in A^2\mid bu=b'u'\}
$$
under the correspondence
$(a,-a')=(b,-b')Q$.

\end{prop}

Proof will be obvious.

\begin{prop}
Suppose that $A$ is a division algebra.
Let $v,v'\in A^2$, $v=(x,y)$, $v'=(x',y')$. Suppose $y=0$, $x\ne 0$.
Then
$Av\cap Av'\ne 0$ if and only if $y'=0, x'\ne 0$.
And in this case
$Av=Av'=A\oplus 0$.
\end{prop}

\begin{proof}
We have 
$$
Av=\{(ax,0)\mid a\in A\}=\{(a,0)\mid a\in A\}=A\oplus 0.
$$
Suppose $Av\cap Av'\ne 0$. There exists $a\in A$ such that 
$ax'\ne 0, \; ay'=0$.
It then follows that  $x'\ne 0$ and $y'=0$.

Conversely if $y'=0$ and $x'\ne 0$, then
$Av'=A\oplus 0=Av$.

\end{proof}

\begin{prop}
Suppose that $A$ is a division algebra.
Let $v,v'\in A^2$, $v=(x,y)$, $v'=(x',y')$. Suppose $y=\lambda x$ for some $\lambda\in F$ and $x\ne 0$.
Then
$Av\cap Av'\ne 0$ if and only if $y'=\lambda x', x'\ne 0$.
In this case
$Av=Av'=\{(a,\lambda a)\mid a\in A\}$.

\end{prop}

\begin{proof}
Let 
$$
P=\begin{pmatrix} 1 & -\lambda \\
0 & 1 
\end{pmatrix}.
$$
Put $w=vP$, $w'=v'P$, so that
$w=(x,0)$, $w'=(x',y'-\lambda x')$.
Then
\begin{align*} 
Av\cap Av'\ne 0
\iff & Aw\cap Aw'\ne 0 \quad(\text{by Proposition 6.1(ii)})\\
\iff &y'-\lambda x'=0, x'\ne 0 \quad(\text{by Proposition 6.2}).
\end{align*}
In this case
$Aw=Aw'=A\oplus 0$.
Then
$Av=Av'=\{(a,\lambda a)\mid a\in A\}$.
\end{proof}

Recall that $v=(x,y)\in A^2$ is said to be degenerate if $x,y$ are linearly dependent over $F$. In this term the proposition is restated: 
 
\begin{cor}
Suppose that $A$ is a division algebra. Let $v,v'\in A^2$ be nonzero. Suppose that $v$ is degenerate.
Then $Av\cap Av' \ne 0$ if and only if $Av=Av'$, in which case $v'$ is also degenerate.

\end{cor}

\begin{prop}
Suppose that $A$ is a division algebra.
Let $v,v'\in A^2$, $v=(x,y)$, $v'=(x',y')$. Suppose $x\ne 0$, $x'\ne 0$ and $x'=\lambda x$ for some $\lambda\in F$.
Then
$Av\cap Av'\ne 0$ if and only if  $y'=\lambda y$.
In this case 
$v'=\lambda v$ and  $Av=Av'$.

\end{prop}

\begin{proof}
By Proposition 6.1(i) we have
$$
Av\cap Av'\ne 0 
\iff \{(a,a')\mid av=a'v'\}\ne 0 ,
$$
and when $v'-\lambda v\ne 0$ we have
$$
Av\cap A(v'-\lambda v)\ne 0\iff
\{(a,a')\mid av=a'(v'-\lambda v)\}\ne 0.
$$
We have
$$
\begin{pmatrix} v \\ v'-\lambda v \end{pmatrix}
=Q\begin{pmatrix} v \\ v' \end{pmatrix}
\quad\text{with}\quad
Q=\begin{pmatrix} 1 & 0 \\ -\lambda & 1 \end{pmatrix}\in GL_2(F).
$$
By Proposition 6.1(iii) we have an isomorphism
$$
\{(a,a')\mid av=a'v'\}\cong
 \{(a,a')\mid av=a'(v'-\lambda v)\}.
$$
Therefore, when $v'-\lambda v\ne 0$, we have
$$
Av\cap Av'\ne 0
\iff 
Av\cap A(v'-\lambda v)\ne 0.
$$
When $y'-\lambda y\ne 0$, by Proposition 6.2 applied to $v'-\lambda v=(0, y'-\lambda y)$
and $v=(x,y)$,  we have
$Av\cap A(v'-\lambda v)\ne 0$ if and only if $x=0, y\ne 0$.
It follows that
if $Av\cap Av'\ne 0$ and $y'\ne \lambda y$ then $x=0$.
Since we are assuming $x\ne 0$, if follows  that
if $Av\cap Av'\ne 0$ then $y'=\lambda y$.
When $y'=\lambda y$, we have $v'=\lambda v$, and as $\lambda\ne 0$, we have
$Av'=Av$.

\end{proof}

\begin{prop} 
Suppose that $A$ is a division algebra. Let $v,v'\in A^2$, $v=(x,y)$, $v'=(x',y')$. 
Suppose that $v$ and $v'$ are nondegenerate. Suppose further that 
$$
(\nu,\nu')\begin{pmatrix} x & y \\ x' & y' \end{pmatrix}
\begin{pmatrix} \lambda \\ \mu \end{pmatrix}
=(0)
$$
with $\lambda, \mu, \nu, \nu'\in F$, $(\lambda,\mu)\ne 0$, $(\nu, \nu')\ne 0$.
Then
$Av\cap Av'\ne 0$ if and only if $\nu v+\nu'v'=0$.
In this case $Av=Av'$.

\end{prop}

\begin{proof}
We may assume $\mu=1$, $\nu'=-1$.
The equation
$$
(\nu,\nu')\begin{pmatrix} x & y \\ x' & y' \end{pmatrix}
\begin{pmatrix} \lambda \\ \mu \end{pmatrix}
=(0)
$$
then says that $\lambda x'+y'=\nu (\lambda x+y)$.
And we have $\lambda x+y\ne 0$, $\lambda x'+y'\ne 0$ by the linear independence. 
Put
$w=(x,\lambda x+y)$, $w'=(x',\lambda x'+y')$.
Then, by Propositions 6.1(ii) and 6.5,  we have
\begin{align*}
Av\cap Av'\ne 0 &\iff Aw\cap Aw'\ne 0\\
&\iff x'=\nu x 
\iff w'=\nu w
\iff v'=\nu v.
\end{align*}
In this case $Av=Av'$. 

\end{proof}

\begin{prop}
Assume  $F$ is  finite.
Let $A$ be a three-dimensional nonassociative division algebra over $F$.
Let $v, v'\in A^2$, 
$v=(x,y)$, $v'=(x', y')$.
Suppose that the $F$-span $\spa{x,y}=Fx+Fy$ is two-dimensional and so are the spans
$\spa{x', y'}$, $\spa{x, x'}$, $\spa{y,y'}$.
If $\spa{x,y}=\spa{x', y'}$, then $Av\cap Av'=0$.
\end{prop}

\begin{proof}
Since the $F$-spans $\spa{x,y}, \spa{x',y'}, \spa{x,x'}, \spa{y,y'}$ are two-dimensional, we have
$$
\spa{x,y}=\spa{x',y'}
\iff 
\dim \spa{x,y,x',y'}=2
\iff 
\spa{x,x'}=\spa{y,y'}.
$$
Assume  $\spa{x,y}=\spa{x',y'}$. Then
$\spa{x,x'}=\spa{y,y'}$, so we write
$$
y=ex+fx', \quad
y'=gx+hx'
$$
with $e,f,g,h\in F$.
We have the isomorphism  of Proposition 6.1(i):
$$
Av \cap Av'\cong \{(a,a')\mid ax=a'x', ay=a'y'\}.
$$
The right multiplication by $x$ on $A$ is denoted by $R_x$.
The equation
$ax=a'x'$ is expressed as $R_{x'}^{-1}R_x(a)=a'$.
And the equation $ay=a'y'$ says $a(ex+fx')=a'(gx+hx')$, or 
$(ea-ga')x=(ha'-fa)x'$, which is expressed as 
$R_{x'}^{-1}R_x(ea-ga')=ha'-fa$.
Putting $Q=R_{x'}^{-1}R_x$, we have
$$
\left\{\begin{aligned}
ax&=a'x',\\
ay&=a'y'
\end{aligned}\right.
\iff
\left\{\begin{aligned}
&Q(a)=a',\\
&Q(ea-ga')=-fa+ha'.
\end{aligned}\right.
$$
By elimination of $a'$ the last equation becomes
$$
Q(ea-gQ(a))=-fa+hQ(a),
$$
that is,
$$
gQ^2(a)+(h-e)Q(a)-fa=0.
$$
Therefore we have an isomorphism
$$
\{(a,a')\mid ax=a'x', ay=a'y'\}
\cong \Ker(gQ^2+(h-e)Q-fI)
$$
given by $(a, Q(a))\leftrightarrow a$.

Let $K/F$ be a cubic extension with generating automorphism $\sigma$.
By Section 1 $A$ is isotopic to the twisted field $(K, \mu)$ associated with an element $c\in K^{\times}$.
Put $\tilde A=K\otimes A$.  
Let $\phi=\phi_{d_0,d_1,d_2}\colon U\times V\to W$ be the split Albert algebra  as defined in Section 2 with base field $K$ and $d_i=-c^{\sigma^i}$:
$$
\phi(\alpha_i,\beta_i)=0,\;
\phi(\alpha_i,\beta_{i+1})=\gamma_{i-1},\;
\phi(\alpha_i,\beta_{i-1})=d_{i+1}\gamma_{i+1}.
$$
By Sections 1 and 2 we have an isomorphism $(f,g,h)$ from the multiplication map $\tilde A\times \tilde A\to \tilde A$ to $\phi$.
By Section 2 the semi-linear automorphism $\sigma\otimes 1$ of $\tilde A$ is translated to a semi-linear automorphism $\lambda$ of $\phi$ such that
$$
\alpha_i\mapsto \alpha_{i+1},\quad \beta_i\mapsto \beta_{i+1},\quad \gamma_i\mapsto \gamma_{i+1}.
$$
Write 
$g(x)=\sum_i x_i \beta_i$, 
$g(x')=\sum_i x'_i \beta_i$
with $x_i, x'_i\in K$.
Then, as $x,x'\in A\subset \tilde A$ are invariant under $\sigma\otimes 1$,
$x_i^{\sigma}=x_{i+1}$,
${x'_i}^{\sigma}=x'_{i+1}$.

Let $\tilde Q=K\otimes Q$. This is a $K$-linear map $\tilde A\to \tilde A$.
The isomorphism $(f,g,h)$ transforms $\tilde Q$ into $R_{g(x')}^{-1}R_{g(x)}\colon U\to U$.
By Proposition 2.1 it follows that the characteristic roots of $\tilde Q$ are
$x_i/x_i'$. 
Suppose $\Ker(gQ^2+(h-e)Q-fI)\ne 0$.
Then $\Ker(g\tilde Q^2+(h-e)\tilde Q-fI)\ne 0$.
Then
$$
g(x_i/x'_i)^2+(h-e)(x_i/x'_i)-f=0
$$
for some $i$.
If the coefficients $g,h-e,f$ are not all zero, then $x_i/x'_i$ has degree less than 3 over $F$.
But $[K:F]=3$, so we must have $x_i/x'_i\in F$.
Then $x_i/x'_i$ are all equal.
Therefore $x=\lambda x'$ with $\lambda\in F$.
This contradicts that $x, x'$ are independent.

So $g,h-e,f$ are all zero.
Then
$y=ex$, 
$y'=ex'$.
This contradicts that $x,y$ are independent.

Consequently
$\Ker (gQ^2+(h-e)Q-fI)=0$,
hence
$\{(a,a')\mid av=a'v'\}=0$. 
Thus $Av\cap Av'=0$.
\end{proof}

\begin{prop}
Assume  $F$ is  finite.
Let $A$ be a three-dimensional nonassociative division algebra over $F$.
Let $v=(x,y), v'=(x',y')\in A^2$.
Suppose that $Av\cap Av'\ne 0$ and $Av\ne Av'$.
Then $\spa{x,y}, \spa{x',y'}, \spa{x,x'}, \spa{y,y'}$ are two-dimensional and
$\spa{x,y}\ne \spa{x',y'}$.

\end{prop}

\begin{proof}
Proposition 6.3 asserts that if $x,y$ are dependent, then either $Av\cap Av'= 0$ or $Av=Av'$.
So $x,y$ must be independent.
Similarly $x',y'$ are independent.
In particular $x, x'$ are nonzero.

Proposition 6.5 asserts that if $x,x'$ are dependent, then either $Av\cap Av'= 0$ or $Av=Av'$.
So $x,x'$ must be independent.
Similar for $y, y'$.

Finally, Proposition 6.7 asserts that if $\spa{x,y}=\spa{x',y'}$, then $Av\cap Av'=0$.
So $\spa{x,y}\ne \spa{x',y'}$.

\end{proof}

\begin{prop}
Assume $F$ is finite.
Let $A$ be a three-dimensional nonassociative division  algebra over $F$.
Let $v, v'\in A^2$, 
$v=(x,y)$, $v'=(x', y')$.
Suppose that the $F$-spans
$\spa{x,y}$, $\spa{x',y'}$, $\spa{x,x'}$, $\spa{y,y'}$ are all two-dimensional.
Let $F_1$ be an algebraic  extension of $F$ and $A_1=F_1\otimes A$. Suppose 
$$
(\nu,\nu')\begin{pmatrix} x & y \\ x' & y' \end{pmatrix}
\begin{pmatrix} \lambda \\ \mu \end{pmatrix}
=(0)
$$
in $A_1$ 
with $\lambda, \mu, \nu, \nu'\in F_1$, $(\lambda,\mu)\ne 0$, $(\nu, \nu')\ne 0$.
Then
$Av\cap Av'=0$.

\end{prop}

\begin{proof}
Replacing $F_1$ by the subfield $F(\lambda,\mu,\nu,\nu')$, we may assume that $F_1/F$ is a finite extension.
By symmetry we may also assume that $\lambda=1$, $\nu=1$.
Then the equation in the hypothesis becomes
\begin{align*}
x+\mu y+\nu' x'+\nu' \mu y'=0. \tag{1}
\end{align*}
Let $\sigma$ be a field automorphism of $F_1$ over $F$.
Apply $\sigma_*=\sigma\otimes 1$ to (1).
\begin{align*}
x+\mu^{\sigma} y+{\nu'}^{\sigma} x'+{\nu'}^{\sigma}\mu^{\sigma}y'=0.\tag 2
\end{align*}
Subtract (1) from (2).
$$
(\mu^{\sigma}-\mu)y+({\nu'}^{\sigma}-\nu')x'+
({\nu'}^{\sigma}\mu^{\sigma}-\nu'\mu)y'=0.
$$
First consider the case where $y,x',y'$ are independent over $F$.
Then 
$\mu^{\sigma}-\mu=0$, ${\nu'}^{\sigma}-\nu'=0$.
Thus $\mu, \nu'$ are invariant under the Galois group of $F_1/F$, so $\mu, \nu'\in F$.
Then by Proposition 6.6 we have 
$Av\cap Av'= 0$ or $v+\nu' v'=0$. 
But the latter would imply that $x, x'$ are dependent.
We must have
$Av\cap Av'=0$.

Next consider the case where $y,x',y'$ are dependent over $F$.
Then $y\in \spa{x',y'}$. By (1) we have $x\in \spa{x',y'}_{F_1}$, hence $x\in \spa{x',y'}_F$.
Then
$\spa{x,y}=\spa{x',y'}$.
By Proposition 6.7 we have $Av\cap Av'=0$.

\end{proof}

\begin{prop}
Assume that $F$ is finite.
Let $A$ be a three-dimensional nonassociative division algebra over $F$.
Let $v, v'\in A^2$, 
$v=(x,y)$, $v'=(x', y')$.
Suppose that the $F$-spans
$\spa{x,y}$, $\spa{x',y'}$, $\spa{x,x'}$, $\spa{y,y'}$ are all two-dimensional.
Let $F_1/F$ be an algebraic extension and $A_1=F_1\otimes A$.
Let $P, Q\in GL_2(F_1)$ and
$$
Q\begin{pmatrix} x & y \\ x' & y' \end{pmatrix} P=
\begin{pmatrix} X & Y \\ X' & Y' \end{pmatrix} 
$$
in $A_1$. Suppose $Av\cap Av'\ne 0$. 
Then
the $F_1$-spans
$\spa{X,Y}_{F_1}$, $\spa{X',Y'}_{F_1}$, $\spa{X,X'}_{F_1}$, 
$\spa{Y, Y'}_{F_1}$ are two-dimensional.

\end{prop}

\begin{proof}
Assume that $Y=\lambda X$ for some $\lambda\in F_1$.
Then
$$
(1,0) \begin{pmatrix} X & Y \\ X' & Y' \end{pmatrix}\begin{pmatrix} \lambda \\ -1 \end{pmatrix}=(0),
$$
namely
$$
(1,0)Q\begin{pmatrix} x & y \\ x' & y' \end{pmatrix} P\begin{pmatrix} \lambda \\ -1 \end{pmatrix}
=(0).
$$
Then, by Proposition 6.9 we have $Av\cap Av'=0$, a contradiction.

Assume next that $X'=\mu X$ for some $\mu\in F_1$.
Then
$$
(\mu, -1)\begin{pmatrix} X & Y \\ X' & Y' \end{pmatrix}\begin{pmatrix} 1 \\ 0 \end{pmatrix}
=(0),
$$
namely
$$
(\mu, -1)Q\begin{pmatrix} x & y \\ x' & y' \end{pmatrix} P \begin{pmatrix} 1 \\ 0 \end{pmatrix}
=(0).
$$
Again by Proposition 6.9 we have $Av\cap Av'=0$, a contradiction.
\end{proof}

\section{Two-dimensional intersection for a split Albert algebra}

Let $F$ be an algebraically closed field.
In this section we show that for the split Albert algebra $\phi_{d_0,d_1,d_2}\colon U\times V\to W$ of Section 2 the existence of a two-dimensional intersection in $W^2$ implies that $d_0d_1d_2=1$.

For the proof we use a normal form of a pair of 2 by 2 matrices, the simplest case  of the Kronecker normal form.
Let $M_2(F)$ denote the algebra of 2 by 2 matrices over $F$.
We say  elements $(A,B)$ and  $(A',B')$ in $M_2(F)\times M_2(F)$ are equivalent if
$(PAQ, PBQ)=(A', B')$ for some $P, Q\in GL_2(F)$.

\begin{prop}
Every  element  in $M_2(F)\times M_2(F)$ is equivalent to one of the elements in (i)--(vii):

(i)  
$$
(\begin{pmatrix} 1 & 0 \\ 0 & 1 \end{pmatrix}, \begin{pmatrix} \lambda & 0 \\ 0 & \mu \end{pmatrix} )
\quad(\lambda,\mu\in F)
$$

(ii)  
$$
(\begin{pmatrix} 1 & 0 \\ 0 & 1 \end{pmatrix}, \begin{pmatrix} \lambda & 0 \\ 1 & \lambda \end{pmatrix})
\quad(\lambda\in F)
$$

(iii) the switch of (i). 

(iv) the switch of (ii).

(v)
$$
(\begin{pmatrix} 1 & 0 \\ 0 & 0 \end{pmatrix},
\begin{pmatrix} 0 & 0 \\ 0 & 1 \end{pmatrix})
$$

(vi)
$$
(\begin{pmatrix} * & * \\ 0 & 0 \end{pmatrix},
\begin{pmatrix} * & * \\ 0 & 0 \end{pmatrix})
$$

(vii)
$$
(\begin{pmatrix} * & 0 \\ * & 0 \end{pmatrix},
\begin{pmatrix} * & 0 \\ * & 0 \end{pmatrix})
$$
\end{prop}
 
\begin{proof}
Let $(A, B)\in M_2(F)\times M_2(F)$.
Suppose first that $A$ is invertible. By left multiplication we can make $A$ into the identity matrix. 
Then, by conjugation we can make $B$ into a Jordan normal form without affecting $A$.
Thus $(A,B)$  falls  in  (i) or (ii).

Suppose next that $A$ and $B$ both have rank one.
By row and column operations we make
$$
A=\begin{pmatrix} 1 & 0 \\ 0 & 0 \end{pmatrix}.
$$
Then write
$$
B=\begin{pmatrix} b_{11} & b_{12} \\ b_{21} & b_{22} \end{pmatrix}.
$$
If $b_{22}\ne 0$, then by row and column operations we can 
make $b_{12}=b_{21}=0$ and $b_{22}=1$ without affecting $A$. Then $(A,B)$ falls in (v).
If $b_{22}=0$, then $b_{12}b_{21}=0$, so $b_{12}=0$ or $b_{21}=0$.
Then $(A,B)$ is of the form in (vi) or (vii).

The proposition readily follows from these considerations.

\end{proof}

Let $d_0,d_1,d_2\in F^{\times}$. 
Let $\phi=\phi_{d_0,d_1,d_2}\colon U\times V\to W$ be the split Albert algebra over $F$ defined in Section 2:
$U$, $V$, $W$ are three-dimensional spaces over $F$ having bases
$(\alpha_i)$, $(\beta_i)$, $(\gamma_i)$, respectively;
With product notation $\phi(u,v)=uv$, we have
$$
\alpha_i\beta_i=0,\quad
\alpha_i\beta_{i+1}=\gamma_{i+2},\quad
\alpha_i\beta_{i+2}=d_{i+1}\gamma_{i+1}.
$$
Put $d=d_0d_1d_2$. We assume $d\ne -1$ throughout.
As in Section 3 we have the induced map $U\times V^2\to W^2\colon (u,(x,y))\mapsto u(x,y)=(ux,uy)$. For $(x,y)\in V^2$ we have the subspace $U(x,y)=\{u(x,y)\mid u\in U\}\subset W^2$.

\begin{thm}
Let $x,y,x',y'\in V$. 
Suppose that the $F$-spans $\spa{x,y}$, $\spa{x',y'}$, $\spa{x,x'}$, $\spa{y,y'}$ are all two-dimensional.  Suppose further that
for any transformation 
$$
Q\begin{pmatrix} x & y \\ x' & y' \end{pmatrix} P=
\begin{pmatrix} X & Y \\ X' & Y' \end{pmatrix} 
$$
with $P, Q\in GL_2(F)$, the $F$-spans
$\spa{X,Y}$, $\spa{X',Y'}$, $\spa{X,X'}$, 
$\spa{Y, Y'}$ remain two-dimensional.
Suppose that $\spa{x,y}\ne \spa{x',y'}$.
If 
$U(x,y)\cap U(x',y')$ is two-dimensional,
then $d=1$.
\end{thm}

\begin{proof}
As observed in Section 2, since $x,y$ are independent, the map
$U\to W^2\colon a\mapsto a(x,y)$ is injective.
Hence, as in Proposition 6.1(i) we have an isomorphism
$$
U(x,y)\cap U(x',y')\cong \{(a,a')\in U^2\mid a(x,y)=a'(x',y')\}.
$$
Write
\begin{alignat*}{2}
x&=x_0 \beta_0+x_1 \beta_1+x_2 \beta_2, \quad&
y&=y_0 \beta_0+y_1 \beta_1+y_2 \beta_2,\\
x'&=x_0' \beta_0+x_1' \beta_1+x_2' \beta_2, \quad&
y'&=y_0' \beta_0+y_1' \beta_1+y_2' \beta_2.
\end{alignat*}
Put
$$
G_0=\begin{pmatrix} x_0 & y_0 \\ x'_0 & y'_0 \end{pmatrix},\;
G_1=\begin{pmatrix} x_1 & y_1 \\ x'_1 & y'_1 \end{pmatrix}.
$$
Take $P, Q\in GL_2(F)$ so that the pair
$(PG_0Q, PG_1Q)$ is  one of (i)--(vii) in Proposition 7.1.
Put 
$$
Q\begin{pmatrix} x & y \\ x' & y' \end{pmatrix} P=
\begin{pmatrix} X & Y \\ X' & Y' \end{pmatrix}. 
$$
By our hypothesis $\spa{X,Y}$, $\spa{X',Y'}$, $\spa{X,X'}$, $\spa{Y,Y'}$ are all two-dimensional.
Note $\spa{x,y,x',y'}=\spa{X,Y,X',Y'}$. This space has dimension $>2$ because 
$\spa{x,y}\ne \spa{x',y'}$. So 
$\spa{X,Y}\ne \spa{X',Y'}$ as well.
And
$$
 \{(a,a')\in U^2\mid a(x,y)=a'(x',y')\}
\cong \{(a,a')\in U^2\mid a(X,Y)=a'(X',Y')\}
$$
by Proposition 6.1(ii), (iii).

Therefore, resetting $X,Y,X',Y'$ as $x,y,x',y'$, we may assume  that $(G_0, G_1)$ itself is  one of  (i)--(vii); then we shall show that if the space
$\{(a,a')\in U^2\mid a(x,y)=a'(x',y')\}$ is two-dimensional,
then $d=1$.

\medskip
Case i: $(G_0,G_1)$ is of the form in (i).
Namely
\begin{alignat*}{2}
x&=\beta_0+x_1 \beta_1+x_2 \beta_2,\quad&
y&=y_2 \beta_2,\\
x'&=x_2' \beta_2, \quad&
y'&=\beta_0+y_1'\beta_1+ y_2'\beta_2.
\end{alignat*}
By the independence of $x, y$ and that of $x',y'$, we must have
$y_2\ne 0$, $x_2'\ne 0$.

Let $a,a'\in U$ and write
$$
a=P\alpha_0+Q\alpha_1+R\alpha_2, \;
a'=P'\alpha_0+Q'\alpha_1+R'\alpha_2
$$
with $P,Q,R, P',Q',R'\in F$.
Expanding $ax$ and $a'x'$ and comparing their coefficients, one sees that the equation $ax=ax'$ amounts to  equations
\begin{align*}
Qx_2+Rx_1d_0&=Q'x'_2, \tag1\\
R +Px_2d_1&=P'x_2'd_1, \tag2\\
Px_1+Q d_2&=0. \tag3
\end{align*}
Similarly the equation $ay=a'y'$ amounts to equations
\begin{align*}
Qy_2&=Q'y_2'+R'y_1'd_0, \tag4\\
Py_2d_1&=R' +P'y_2'd_1, \tag5\\
0&=P'y_1'+Q' d_2. \tag6
\end{align*}
Solving (2), (5), (3), (6), we obtain
\begin{align*}
R&=-Px_2d_1+P'x_2' d_1, \tag7\\
R'&= Py_2d_1-P'y_2' d_1, \tag8\\
Q&=-P x_1 \frac{1}{d_2}, \tag9\\
Q'&=-P'y_1'\frac{1}{d_2}. \tag{10}
\end{align*}
Putting (7), (8), (9), (10) into (1) and (4), we obtain
\begin{align*}
-P(x_1 x_2\frac{1}{d_2}+x_2 x_1d_1d_0)
&=-P'(y_1' x_2'\frac{1}{d_2} +x_2'x_1 d_1d_0),\\
-P(x_1 y_2\frac{1}{d_2}+y_2 y_1'd_1d_0)
&=-P'(y_1' y_2'\frac{1}{d_2}+y_2' y_1'd_1 d_0).
\end{align*}
Multiplying the both sides by $d_2$ and using $d=d_0d_1d_2$, we have
\begin{align*}
Px_1x_2(1+d)&=P'x_2'(y_1'+x_1d), \tag{11}\\
Py_2(x_1+y_1'd)&=P'y_1'y_2'(1+d). \tag{12}
\end{align*}
Thus (1)--(6) are equivalent to (7)--(12), so that we have an isomorphism
\begin{align*}
&\{(a,a')\in U^2\mid ax=a'x', ay=a'y'\}\\
&\cong
 \{(P,Q,R,P',Q',R')\in F^6\mid \text{(7)--(12) hold}\}.
\end{align*}
Let $M$ be the coefficient matrix of (11), (12):
$$
M=\begin{pmatrix}
x_1x_2(1+d) & x_2'(y_1'+x_1d) \\
y_2(x_1+y_1'd) & y_1'y_2'(1+d)
\end{pmatrix}.
$$
Then 
$$
\dim\{(a,a')\in U^2\mid ax=a'x', ay=a'y'\}
=2-\rank M.
$$
So
$\{(a,a')\in U^2\mid ax=a'x', ay=a'y'\}$ is two-dimensional if and only if $M=O$.

Suppose $M=O$.
Since $y_2\ne 0$, $x_2'\ne0$, $1+d\ne 0$, we have
\begin{alignat*}{2}
x_1x_2&=0, \quad& 
y_1'+x_1d&=0, \\
x_1+y_1'd&=0,\quad&
y_1'y_2'&=0.
\end{alignat*}

Case where $x_1=0$, $y_1'=0$.
The four equations are all trivial.

Case where $x_1\ne 0$, $y_1'\ne 0$.
Then
$$
x_2=0, \;y_2'=0, \; 
x_1=-y_1'd=x_1d^2.
$$
Hence
$d^2=1$.
Since $d\ne -1$, we have $d=1$.

Case where $x_1=0$, $y_1'\ne 0$.
This contradicts the second of the four equations.

Case where $x_1\ne 0$, $y_1'=0$.
This contradicts the third equation.

Consequently we have
$M=O$ if and only if
$x_1=0,\; y_1'=0$
 or 
$d=1,\; x_2=0,\; y_2'=0,\; x_1=-y_1'$. 
In the former case we have
\begin{alignat*}{2}
x&=  \beta_0+x_2\beta_2, \quad&
y&=y_2\beta_2,\\
x'&=x_2'\beta_2, \quad&
y'&= \beta_0+y_2'\beta_2.
\end{alignat*}
Then $\spa{x,y}=\spa{\beta_0,\beta_2}=\spa{x',y'}$, contrary to the hypothesis.

We conclude that 
if $M=O$ then  $d=1$.

\medskip
Case ii: $(G_0,G_1)$ is of the form in  (ii). Namely
\begin{alignat*}{2}
x&=  \beta_0+x_1 \beta_1+x_2 \beta_2, \quad&
y&=y_2 \beta_2,\\
x'&= \beta_1+x_2'\beta_2, \quad&
y'&= \beta_0+y_1'\beta_1+y_2'\beta_2
\end{alignat*}
with $x_1=y_1'$.
By the independence of $x, y$ we must have
$y_2\ne 0$.

Let $a,a'\in U$ and write
$$
a=P\alpha_0+Q\alpha_1+R\alpha_2, \;
a'=P'\alpha_0+Q'\alpha_1+R'\alpha_2.
$$
Similarly to the previous case one sees that 
the equations $ax=a'x'$ and $ay=a'y'$ amount to  equations
\begin{align*}
R&=-P x_2 d_1+P' x_2'  d_1,  \\
R'&= P y_2 d_1-P' y_2' d_1,  \\
Q&= -P x_1 \frac{1}{d_2}+P' \frac{1}{d_2},  \\
Q'&=-P' y_1' \frac{1}{d_2}, \\
P( x_1x_2 (1+d)+ y_2  d)
&=P'(  x_2 + y_1'x_2' +( x_2'x_1 + y_2'  )d),\\
P( x_1y_2 + y_2y_1' d)
&=P'(  y_2 + y_1'y_2' (1+d)).
\end{align*}
Put
$$
M=\begin{pmatrix}
 x_1x_2 (1+d)+ y_2   d & 
  x_2 + y_1'x_2' +( x_2'x_1 + y_2'  )d
\\
 x_1y_2 + y_2y_1' d & 
  y_2 + y_1'y_2' (1+d)
\end{pmatrix}.
$$
We have 
$\{(a,a')\in U^2\mid ax=a'x', ay=a'y'\}$ is two-dimensional if and only if $M=O$.

Recall that $x_1=y_1'$.
The $(2,1)$-entry of $M$ is
$x_1 y_2+y_2y_1'd=y_2x_1(1+d)$.
Since $y_2\ne 0$ and $1+d\ne 0$, if this entry is zero, then
$x_1=0$. But when $x_1=0$, 
the $(1,1)$-entry is
$y_2 d\ne 0$
as $y_2\ne 0$.

We conclude that $M\ne O$.
This settles  Case ii.

\medskip
 Case iii \& iv: $(G_0,G_1)$ is of the form in (iii) or (iv).

Let $\pi$ be the permutation $0\mapsto 1, 1\mapsto 0, 2\mapsto 2$.
As noted in Section 2 we have an isomorphism of bilinear maps
$$
(f,g,h)\colon (\phi_{d_0,d_1,d_2}\colon U\times V\to W)\to 
(\phi_{1/d_1,1/d_0,1/d_2}\colon U\times V\to W)
$$
given by
$$
f(\alpha_i)=\alpha_{\pi(i)}, \;
g(\beta_i)=\beta_{\pi(i)},\;
h(\gamma_i)=d_i^{-1}\gamma_{\pi(i)}.
$$
This isomorphism interchanges $G_0$ and $G_1$, so turns  Case iii into Case i, Case iv into Case ii.
And $d_0d_1d_2=1$ if and only if $d_0^{-1}d_1^{-1}d_2^{-1}=1$. This settles the present case.

\medskip
Case v: $(G_0,G_1)$ is of the form in (v). Namely
\begin{alignat*}{2}
x&=\beta_0+x_2\beta_2\,\quad&
y&=y_2\beta_2,\\
x'&=x_2'\beta_2,\quad&
y'&=\beta_1+y_2'\beta_2.
\end{alignat*}
We have $x_2'\ne 0$, $y_2\ne 0$.

Let $a,a'\in U$ and write
$$
a=P\alpha_0+Q\alpha_1+R\alpha_2, \;
a'=P'\alpha_0+Q'\alpha_1+R'\alpha_2.
$$
The equations $ax=a'x'$ and $ay=a'y'$ amount to  equations
\begin{align*}
Qx_2&=Q'x'_2, \\
R+Px_2d_1&=P'x_2'd_1, \\
Qd_2&=0, \\
Qy_2&=Q'y_2'+R'd_0, \\
Py_2d_1&=P'y_2'd_1 \\
0&=P'.
\end{align*}
From these one readily deduces that
$P,Q,R, P', Q', R'$ are all zero.
We conclude 
$\{(a,a')\in U^2\mid ax=a'x', ay=a'y'\}=0$.

\medskip
Case vi: $(G_0,G_1)$ is of the form in (vi).
Then
$x'=x_2' \beta_2$, $y'=y_2'\beta_2$.
So $x', y'$ are dependent.

\medskip
Case vii: $(G_0, G_1)$ is of the form in (vii).
Then
$y=y_2\beta_2$, $y'=y_2'\beta_2$.
So $y, y'$ are dependent.

In every possible case  we have proved that if $\{(a,a')\in U^2\mid ax=a'x', ay=a'y'\}$ is two-dimensional then $d=1$.

\end{proof}

\section{Two-dimensional intersection for a division algebra}

Let $F$ be a finite field. Let $A$ be a three-dimensional nonassociative division algebra over $F$.

\begin{thm}
Suppose that $\dim(Av\cap Av')=2$ for some $v,v'\in A^2$. Then $A$ is isotopic to a commutative algebra.
\end{thm}

\begin{proof}
Let $\tilde F$ be an algebraic closure of $F$. Let $\tilde A=\tilde F\otimes A$.
Write $v=(x,y)$, $v'=(x',y')$.
By Proposition 6.8 the $F$-spans
$\spa{x,y}$, $\spa{x',y'}$, $\spa{x,x'}$, $\spa{y,y'}$ are all two-dimensional and $\spa{x,y}\ne \spa{x',y'}$.
By Proposition 6.10, after the transformation
$$
Q\begin{pmatrix} x & y \\ x' & y' \end{pmatrix}
P=
\begin{pmatrix} X & Y \\ X' & Y' \end{pmatrix}
$$
for any $P, Q\in GL_2(\tilde F)$,
the $\tilde F$-spans $\spa{X,Y}$, $\spa{X',Y'}$, $\spa{X,X'}$, $\spa{Y,Y'}$ in $\tilde A$ remain two-dimensional.

Let $m\colon A\times A\to A$ denote the multiplication of $A$ and 
$\tilde m\colon \tilde A\times \tilde A\to \tilde A$ that of $\tilde A$.
Let $K\subset \tilde F$ be a cubic extension of $F$ and 
$\sigma$ a generating automorphism of $K/F$.
By Theorem 1.1 $A$ is isotopic to the twisted field $(K,\mu)$ associated with an element $c\in K^{\times}$. 
And we have an isomorphism  $K\otimes (K,\mu)\to (K^3,\nu)$ as defined in Section 1.
The multiplication $\nu\colon K^3\times K^3\to K^3$ is viewed as a bilinear map $\phi=\phi_{d_0,d_1,d_2}\colon U\times V\to W$ with 
$d_i=-c^{\sigma^i}$ (Section 2).
Let $\tilde \phi\colon \tilde U\times \tilde V\to \tilde W$ be the map obtained from $\phi$ by the scalar extension $\tilde F/K$.

Combining these isomorphisms, 
we have an isomorphism $(f,g,h)$ from the bilinear map $\tilde m\colon \tilde A\times \tilde A\to \tilde A$ to the bilinear map $\tilde \phi\colon 
\tilde U\times \tilde V\to \tilde W$.
Now we take $\tilde F$ as a base field and apply Theorem  7.2 to the elements $g(x), g(y), g(x'), g(y')\in \tilde V$.
The consequence is that $d=d_0d_1d_2=1$, namely $N(c)=-1$.
Then $(K, \mu)$ is isotopic to the twisted field associated with $-1\in K^{\times}$, which is  commutative.
Hence $A$ is isotopic to a commutative algebra.

\end{proof}

\section{Intersection of a given dimension}

Throughout this section $F$ is a finite field and $A$ is a three-dimensional nonassociative division algebra over $F$. 
We showed in Section 4 that two-dimensional intersections of $Av$ occur when $A$ is isotopic to a commutative algebra. In this section we look at intersections of $Av$ more closely. Especially, given $v\in A^2$ we compute the number of $v'\in A^2$ such that $\dim(Av\cap Av')=0,1,2$, respectively. 

\subsection{the commutative case}

In this subsection we treat the case where $A$ is commutative.
First we supplement Theorem 4.1 with uniqueness assertion.  

\begin{prop}
Let $v=(x,y)\in A^2$ be nondegenerate.
Let $x'\in A$ and suppose that $x, x'$ are independent over $F$.
Then there exists a unique $y'\in A$ such that $\dim(Av\cap Av')=2$ for $v'=(x',y')$.
And for such $y'$ we have 
$Av\cap Av'=\spa{x',y'}v=\spa{x,y}v'$.
\end{prop}

\begin{proof}
The existence was shown in the proof of Theorem 4.1: if $y'\in A$ is taken so that $x'y=xy'$, then
$Av\cap Av'=\spa{x',y'}v=\spa{x,y}v'$ for $v'=(x',y')$ and $\dim(Av\cap Av')=2$.

Let us prove the uniqueness.
Suppose that we have $y_1', y_2'\in A$ such that
$\dim(Av\cap Av_1')=2$, $\dim(Av\cap Av_2')=2$
for $v_1'=(x',y_1')$, $v_2'=(x',y_2')$.
Since $\dim Av=3$, we must have
$(Av\cap Av_1')\cap (Av\cap Av_2')\ne 0$,
so 
$Av_1'\cap Av_2'\ne 0$.
But $v_1'$ and $v_2'$ have the same $x$-coordinate. Therefore, by Proposition 6.5  
they have the same $y$-coordinate, that is,  $y_1'=y_2'$.

\end{proof}

\begin{prop}
Let $v=(x,y), v'=(x',y')\in A^2$. Suppose $\dim(Av\cap Av')=2$. Then 
$Av\cap Av'=\spa{x',y'}v=\spa{x,y}v'$.
\end{prop}

\begin{proof}
By Proposition 6.8 $x,y$ are independent and $x,x'$ are independent.
By  Proposition 9.1  we have
$Av\cap Av'=\spa{x',y'}v=\spa{x,y}v'$.
\end{proof}

\begin{prop}
Let $v, v_1', v_2'\in A^2$. Suppose $\dim (Av\cap Av_1')=\dim(Av\cap Av_2')=2$. If $Av\cap Av_1'=Av\cap Av_2'$, then $Av_1'=Av_2'$.

\end{prop}

\begin{proof}
Write $v_1'=(x_1',y_1')$, $v_2'=(x_2', y_2')$.
Since $\dim(Av\cap Av_1')=2$ and $\dim(Av\cap Av_2')=2$, we have by Proposition 9.2 
that $Av\cap Av_1'=\spa{x_1',y_1'}v$ and
$Av\cap Av_2'=\spa{x_2',y_2'}v$.
Suppose $Av\cap Av_1'=Av\cap Av_2'$. Then
$\spa{x_1',y_1'}=\spa{x_2',y_2'}$.
Also $Av_1'\cap Av_2'\ne 0$.
If $Av_1'\ne Av_2'$, Proposition 6.8 would imply that $\spa{x_1',y_1'}\ne\spa{x_2',y_2'}$. So we must have $Av_1'=Av_2'$.

\end{proof}

\begin{prop}
Let $v=(x,y),v'=(x',y')\in A^2$ be nondegenerate. Suppose that $v,v'$ are independent.
Let $a,a'\in A$ with $a\ne 0$, $a'\ne 0$. Suppose $av=a'v'$.
If $a'\in \spa{x,y}$, then $\dim(Av\cap Av')=2$.
\end{prop}

\begin{proof}
Suppose $a'=\lambda x+\mu y$ with $\lambda,\mu\in F$.
Take  a matrix
$$
P=\begin{pmatrix} \lambda & * \\ \mu & *\end{pmatrix} \in GL_2(F).
$$
Then 
$(x,y)P=(a', *)$.
Put 
$\tilde v=vP$,
$\tilde v'=v'P$.
Then $\tilde v$ and $\tilde v'$ are nondegenerate;
$\tilde v, \tilde v'$ are independent;
$a\tilde v=a'\tilde v'$; 
$Av\cap Av'\cong A\tilde v\cap A\tilde v'$.

Replacing $v, v'$ by $\tilde v, \tilde v'$, we may assume $a'=x$ from the beginning.
Then the equation $av=a'v'$ reads
$ax=xx'$, $ay=xy'$,
hence
$a=x'$, $x'y=xy'$.
Then, as in the proof of Theorem 5.1 we have
$y'v=yv'$ and
$\spa{x',y'}v=\spa{x,y}v'\subset Av\cap Av'$.
Since $v, v'$ are independent, we have $Av\ne Av'$.
So
$Av\cap Av'=\spa{x',y'}v=\spa{x,y}v'$.
This is two-dimensional.
\end{proof}

\begin{prop}
Let $v=(x,y), v'=(x',y')\in A^2$. Suppose $\dim (Av\cap Av')=2$. Let $a,a'\in A$ 
with $a\ne 0$, $a'\ne 0$. Suppose $av=a'v'$.
Then $a'\in \spa{x,y}$.
\end{prop}

\begin{proof}
By Proposition 9.2 we have
$Av\cap Av'=\spa{x,y}v'$.
Then $a'v'\in \spa{x,y}v'$, hence $a'\in \spa{x,y}$.

\end{proof}

\begin{prop}
Let $v=(x,y)\in A^2$ be nondegenerate and $v'\in A^2$. Let $a, a'\in A$ with $a\ne 0$, $a'\ne 0$.
Suppose $av=a'v'$.

(i) $\spa{a}=\spa{a'}\iff \spa{v}=\spa{v'}\iff Av=Av'$.

(ii) Suppose $\spa{a}\ne \spa{a'}$. 
If $a'\in \spa{x,y}$ then $\dim(Av\cap Av')=2$, and 
if $a'\notin \spa{x,y}$ then $Av\cap Av'=\spa{av}$.

\end{prop}

\begin{proof}
Since $0\ne av=a'v'\in Av\cap Av'$, we have $Av\cap Av'\ne 0$, so 
$v'$ is nondegenerate by Corollary 6.4.

(i) The fist equivalence is clear, the second due to Theorem 4.1. 

(ii) Let $\spa{a}\ne \spa{a'}$. Then $\spa{v}\ne \spa{v'}$, $Av\ne Av'$.
Proposition 9.4 says if $a'\in \spa{x,y}$ then $\dim(Av\cap Av')=2$, 
while Proposition 9.5 says 
if $\dim (Av\cap Av')=2$ then $a'\in \spa{x,y}$.
So 
$a'\in \spa{x,y}$ if and only if $\dim(Av\cap Av')=2$.
Hence
$a'\notin \spa{x,y}$ if and only if $\dim(Av\cap Av')=1$, in which case $Av\cap Av'=\spa{av}$.
This proves (ii).
\end{proof}

Let $q$ denote the number of elements of the finite field $F$.
The following proposition is valid irrespective of the commutativity of $A$.

\begin{prop}
We have
\begin{align*}
&\#\{v\in A^2\mid \text{$v$ is nondegenerate}\}=(q^3-1)(q^3-q),\\
&\#\{v\in A^2\mid \text{$v$ is degenerate and nonzero}\}=(q^3-1)(q+1),
\end{align*}
and
\begin{align*}
&\#\{Av\mid \text{$v\in A^2$ is nondegenerate}\}=(q^3-1)(q+1)q,\\
&\#\{Av\mid \text{$v\in A^2$ is degenerate and nonzero}\}=q+1.
\end{align*}
\end{prop}

\begin{proof}
The number of $(x,y)\in A^2$ such that $x,y$ are independent is 
 $(q^3-1)(q^3-q)$. 
The number of $(x,y)\ne 0$ such that $x,y$ are dependent is
$(q^6-1)-(q^3-1)(q^3-q)=(q^3-1)(q+1)$.
This proves the first two equalities.

Owing to Theorem 4.1 and Corollary 6.4  the last two follow by division by
$\#F-1=q-1$ and $\#A-1=q^3-1$. 

\end{proof}
 
\begin{prop}
Let $v=(x,y)\in A^2$ be nondegenerate. Then
$$
\#\{v'\in A^2\mid \dim(Av\cap Av')=2\}=q^3-q.
$$
\end{prop}

\begin{proof}
By Proposition 6.8 and Proposition 9.1 we have a bijection
$$
\{v'\in A^2\mid \dim(Av\cap Av')=2\} \to A-\spa{x}
$$
taking $v'=(x',y')$ to $x'$.
It follows that
$$
\#\{v'\in A^2\mid \dim(Av\cap Av')=2\}=\#(A-\spa{x})=q^3-q.
$$
\end{proof}

\begin{prop}
Let $v=(x,y)\in A^2$ be nondegenerate.
We have a bijection
$$
\{Av'\mid v'\in A^2, \dim(Av\cap Av')=2\}\to \{M\subset Av\mid \dim M=2, M\ne \spa{x,y}v\}
$$
taking $Av'$ to  $Av\cap Av'$.

\end{prop}

\begin{proof}
By Proposition 9.2 and Proposition 6.8 
if $v'=(x',y')\in A^2$ and
$\dim(Av\cap Av')=2$ then $Av\cap Av'=\spa{x',y'}v\ne \spa{x,y}v$.
Therefore we have a map
$$
\{Av'\mid v'\in A^2, \dim(Av\cap Av')=2\}\to \{M\subset Av\mid \dim M=2, M\ne \spa{x,y}v\}
$$
taking $Av'$ to  $Av\cap Av'$.
By Proposition 9.3 this map is injective.
Its target  has cardinality 
$(q^3-1)/(q-1)-1=q^2+q$,
and the domain has cardinality
$(q^3-q)/(q-1)=q^2+q$ by the preceding proposition.
Hence the map is bijective.

\end{proof}

\begin{prop}
Let $v=(x,y)\in A^2$ be nondegenerate.

(i) For any one-dimensional subspace $L\subset Av$ we have
$$
\#\{v'\in A^2\mid Av\cap Av'=L\}=\begin{cases}
q^3-q^2 \quad &\text{if $L\subset \spa{x,y}v$, } \\
(q^2-1)(q-1) &\text{if $L\not\subset \spa{x,y}v$.}
\end{cases}
$$
(ii)
$$
\#\{v'\in A^2\mid \dim(Av\cap Av')=1\}
=q^3(q^2-1).\\
$$
\end{prop}

\begin{proof}
(i) Let $L=\spa{av}$ with $a\in A$, $a\ne 0$.
By Proposition 9.6  we have a bijection
$$
\{v'\in A^2\mid Av\cap Av'=\spa{av}\}
\cong 
\{a'\in A\mid a'\notin  \spa{a}, \; a'\notin \spa{x,y}\},
$$
in which $v'$ and $a'$ are related  by the equation $av=a'v'$.
If $a\in \spa{x,y}$, the righthand set equals
$A-\spa{x,y}$,
which has cardinality $q^3-q^2$. 
If $a\notin \spa{x,y}$, the righthand set equals
$A-(\spa{x,y}\cup \spa{a})$,
which has cardinality
$q^3-(q^2+q-1)=(q-1)(q^2-1)$.
It follows that
$$
\#\{v'\in A^2\mid Av\cap Av'=\spa{av}\}=\begin{cases}
q^3-q^2 \quad &\text{if $a\in \spa{x,y}$, } \\
(q^2-1)(q-1) &\text{if $a\notin \spa{x,y}$.}
\end{cases}
$$
This proves (i).

(ii)
The number of one-dimensional subspaces $L\subset Av$ contained in $\spa{x,y}v$ is $q+1$, and that of $L$ not contained in $\spa{x,y}v$ is $q^2$. 
Using (i), we compute
\begin{align*}
\#\{v'\in A^2\mid \dim(Av\cap Av')=1\}
&=(q^3-q^2)(q+1)+(q^2-1)(q-1)q^2\\
&=q^3(q^2-1).
\end{align*}
\end{proof}

\begin{prop}
Let $v\in A^2$ be nondegenerate.
We have
\begin{align*}
&\#\{v'\in A^2\mid \text{$v'$ is nondegenerate and $Av\cap Av'=0$}\}\\
&=(q-1)(q^5-q^3-2q^2-2q-1).
\end{align*}
\end{prop}

\begin{proof}
We know
\begin{align*}
&\#\{v'\in A^2\mid \text{$v'$ is nondegenerate}\}
=(q^3-1)(q^3-q),\\
&\#\{v'\in A^2\mid \dim(Av\cap Av')=3\}=q-1,\\
&\#\{v'\in A^2\mid \dim(Av\cap Av')=2\}=q^3-q,\\
&\#\{v'\in A^2\mid \dim(Av\cap Av')=1\}=q^3(q^2-1).
\end{align*}
Also 
$Av\cap Av'\ne 0$ only if $v'$ is nondegenerate by Corollary 6.4

Therefore we have
\begin{align*}
&\#\{v'\in A^2\mid \text{$v'$ is nondegenerate and $Av\cap Av'=0$}\}\\
&=(q^3-1)(q^3-q)-\{(q-1)+(q^3-q)+q^3(q^2-1)\}\\
&=(q-1)(q^5-q^3-2q^2-2q-1).
\end{align*}
\end{proof}

\begin{prop}
Let $0\ne v\in A^2$.
If $v$ is nondegenerate, 
$$
\#\{Av'\mid v'\ne 0, Av\cap Av'=0\}=
q^5-q^3-2q^2-q.
$$
If $v$ is degenerate,
$$
\#\{Av'\mid v'\ne 0, Av\cap Av'=0\}=
q^2(q^3+q^2-1).
$$
\end{prop}

\begin{proof}
Suppose first that $v$ is nondegenerate.
Then $Av\cap Av'=0$ for all degenerate $v'$ by Corollary 6.4.
We have by Proposition 9.7
$$
\#\{Av'\mid v'\ne 0 , \text{$v'$ is degenerate}\}=q+1.
$$
And by Proposition 9.11
$$
\#\{Av'\mid \text{$v'$ is nondegenerate and $Av\cap Av'=0$}\}=
q^5-q^3-2q^2-2q-1.
$$
Therefore
\begin{align*}
\#\{Av'\mid v'\ne 0, Av\cap Av'=0\}
&=(q+1)+(q^5-q^3-2q^2-2q-1)\\
&=q^5-q^3-2q^2-q.
\end{align*}

Suppose next that $v$ is degenerate.
Then $Av\cap Av'=0$ for all nondegenerate $v'$.
By Proposition 9.7
$$
\#\{Av'\mid \text{$v'$ is nondegenerate}\}=(q^3-1)(q+1)q.
$$
For a degenerate $v'$ we have $Av\cap Av'=0$ if $Av\ne Av'$.
By Proposition 9.7
$$
\#\{Av'\mid v'\ne 0 , \text{$v'$ is degenerate}, Av\ne Av'\}=q.
$$
Therefore
\begin{align*}
\#\{Av'\mid v'\ne 0, Av\cap Av'=0\}
&=(q^3-1)(q+1)q+q\\
&=q^2(q^3+q^2-1).
\end{align*}

\end{proof}

\subsection{the noncommutative case}
 
We next treat the case where $A$ is not isotopic to a commutative algebra.
By theorem 8.1 it never occurs that  $\dim(Av\cap Av')=2$ for $v,v'\in A^2$.

\begin{prop}
Let $v=(x,y)\in A^2$ be nondegenerate and $v'\in A^2$.
Let $a,a'\in A$ with $a\ne 0$, $a'\ne 0$. Suppose $av=a'v'$.

(i) $\spa{a}=\spa{a'}\iff \spa{v}=\spa{v'}\iff Av=Av'$.

(ii) When $\spa{a}\ne \spa{a'}$, we have
$Av\cap Av'=\spa{av}$.
\end{prop}

\begin{proof}
(i) is proved as in the commutative case.

(ii)
Let $\spa{a}\ne \spa{a'}$. Then $\spa{v}\ne \spa{v'}$, $Av\ne Av'$ by (i).
And $0\ne av=a'v'\in Av\cap Av'$. So $\dim(Av\cap Av')=1$, hence
$Av\cap Av'=\spa{av}$.
\end{proof}

\begin{prop}
Let $v=(x,y)\in A^2$ be nondegenerate.

(i) For any one-dimensional subspace $L\subset Av$ we have
$$
\#\{v'\in A^2\mid Av\cap Av'=L\}
=q^3-q.
$$

(ii)
$$
\#\{v'\in A^2\mid \dim (Av\cap Av')=1\}
=q(q+1)(q^3-1).
$$
\end{prop}

\begin{proof}
(i)
Let $L=\spa{av}$ with $a\in A$. By the preceding proposition we have a bijection
$$
\{v'\in A^2\mid Av\cap Av'=\spa{av}\}
\cong \{a'\in A\mid a'\notin \spa{a}\},
$$
in which $v'$ corresponds to $a'$ if $av=a'v'$.
This set has cardinality $q^3-q$.

(ii)
The number of one-dimensional subspaces $L\subset Av$ is $q^2+q+1$.
It follows from (i) that
\begin{align*}
&\#\{v'\in A^2\mid \dim (Av\cap Av')=1\}\\
&=(q^3-q)(q^2+q+1)=q(q+1)(q^3-1).
\end{align*}

\end{proof}

\begin{prop}
Let $v\in A^2$ be nondegenerate. We have
\begin{align*}
&\#\{v'\in A^2\mid \text{$v'$ is nondegenerate and $Av\cap Av'=0$}\}\\
&=(q-1)(q^5 - 2 q^3 - 3 q^2 - 2 q - 1).
\end{align*}
\end{prop}

\begin{proof}
We know
\begin{align*}
&\#\{v'\in A^2\mid \text{$v'$ is nondegenerate}\}
=(q^3-1)(q^3-q),\\
&\#\{v'\in A^2\mid \dim(Av\cap Av')=3\}=q-1,\\
&\#\{v'\in A^2\mid \dim(Av\cap Av')=2\}=0,\\
&\#\{v'\in A^2\mid \dim(Av\cap Av')=1\}=q(q+1)(q^3-1).
\end{align*}
And
$Av\cap Av'\ne 0$ only if $v'$ is nondegenerate.
It follows that
\begin{align*}
&\#\{v'\in A^2\mid \text{$v'$ is nondegenerate and $Av\cap Av'=0$}\}\\
&=(q^3-1)(q^3-q)-\{(q-1)+q(q+1)(q^3-1)\}\\
&=(q-1)(q^5 - 2 q^3 - 3 q^2 - 2 q - 1).
\end{align*}

\end{proof}

\begin{prop}
Let $0\ne v\in A^2$.  
If $v$ is nondegenerate, 
$$
\#\{Av'\mid v'\ne 0, Av\cap Av'=0\}=q^5-2q^3-3q^2-q.
$$
If $v$ is degenerate,
$$
\#\{Av'\mid v'\ne 0, Av\cap Av'=0\}=
q^2(q^3+q^2-1).
$$
\end{prop}

\begin{proof}
Suppose  that $v$ is nondegenerate.
Then $Av\cap Av'=0$ for all degenerate $v'$.
We have by Proposition 9.7
$$
\#\{Av'\mid v'\ne 0 , \text{$v'$ is degenerate}\}=q+1.
$$
And by Proposition 9.15
$$
\#\{Av'\mid \text{$v'$ is nondegenerate and $Av\cap Av'=0$}\}=
q^5 - 2 q^3 - 3 q^2 - 2 q - 1.
$$
Therefore
\begin{align*}
\#\{Av'\mid v'\ne 0, Av\cap Av'=0\}
&=(q+1)+(q^5-2q^3-3q^2-2q-1)\\
&=q^5-2q^3-3q^2-q.
\end{align*}

The proof for the degenerate case is the same as that of Proposition 9.12. 

\end{proof}

\end{document}